\documentclass[12pt]{amsart}
\usepackage[colorlinks=true,linkcolor=blue,anchorcolor=blue,linktocpage=true,urlcolor=blue,citecolor=blue]{hyperref} 
\usepackage{amssymb}
\usepackage{mathrsfs}
\usepackage{enumerate}
\usepackage{hyperref}
\usepackage{color}
\newtheorem{theorem}{Theorem}[section]
\newtheorem{definition}{Definition}[section]
\newtheorem{proposition}[theorem]{Proposition}
\newtheorem{lemma}[theorem]{Lemma}
\newtheorem{corollary}[theorem]{Corollary}
\newtheorem*{claim}{Claim}

\newtheorem{theoremalph}{Theorem}

 \def\JJ{{\mathbb J}} 

 \def\NN{{\mathbb N}}

 \def\ZZ{{\mathbb Z}}

\def \G{\Gamma}
\def \D{\Delta}

\def \cC{\mathcal{C}}
\def \cS{\mathcal{S}}

\def \F{\EuScript{F}}

\def \B{\mathcal{B}}
\def \H{\mathcal{H}}

\def \K{\mathcal{K}}

\def \cM{\mathscr{M}}

\def\F {{\mathcal F}}
\def\B {{\mathcal B}}
\def\P{\mathcal P}

\def\G{{\mathcal G}}

\begin{document}

\title[Physical measures on partially hyperbolic diffeomorphisms]{Physical measures on partially hyperbolic diffeomorphisms with multi 1-D centers}

\author[Zeya Mi]{Zeya Mi}

\address{School of Mathematics and Statistics,
Nanjing University of Information Science and Technology, Nanjing 210044, Jiangsu, China}
\email{\href{mailto:mizeya@163.com}{mizeya@163.com}}

\author[Yongluo Cao]{Yongluo Cao}

\address{Department of Mathematics, Soochow University, Suzhou 215006, Jiangsu, China.}
\address{Center for Dynamical Systems and Differential Equation, Soochow University, Suzhou 215006, Jiangsu, China}

\email{\href{mailto:ylcao@suda.edu.cn}{ylcao@suda.edu.cn}}

\thanks{Y. Cao is the corresponding author. Z. Mi\ was partially supported by National Key R\&D Program of China(2022YFA1007800) and NSFC 12271260. Y. Cao\ was partially supported by National Key R\&D Program of China(2022YFA1005802) and NSFC 12371194.}

\date{\today}


\begin{abstract}
In this paper, we study physical measures for partially hyperbolic diffeomorphisms with multi one-dimensional centers under the condition that all Gibbs $u$-states are hyperbolic. We prove the finiteness of ergodic physical measures. Then by building a criterion for the basin covering property of physical measures, 
we obtain the basin covering property for ergodic physical measures when there exists some limit measure of empirical measures for Lebesgue almost every point that admits the same sign of Lyapunov exponents on each center.
\end{abstract}

\maketitle


\section{Introduction}
A central topic in dynamical systems is to describe the typical behavior of orbits as time goes to infinity, with various viewpoints.  
%
One way to formulate this behavior is through the Sinai-Ruelle-Bowen(SRB) theory, which also links highly to the important objects in dynamical systems, such as entropies, Lyapunov exponents, thermodynamical formalism.

Let $f$ be a $C^2$ diffeomorphism on a compact Riemannian manifold $M$.
By \emph{physical measure} one means an $f$-invariant measure $\mu$ whose basin
$$
\B(\mu,f)=\Big\{x\in M: \lim_{n\to +\infty}\frac{1}{n}\sum_{i=0}^{n-1}\varphi(f^i(x))=\int \varphi {\rm d}\mu,\quad \forall\, \varphi\in C(M)\Big\}
$$
exhibits positive Lebesgue measure. In 1970's, Sinai, Ruelle and Bowen \cite{Sin72, Rue76, Bow75, BoR75} showed that for uniformly hyperbolic systems, there are finitely many physical measures, whose basins cover a full Lebesgue measure subset, called \emph{basin covering property}. 

In general, we say an invariant measure $\mu$ is an \emph{SRB} measure if it admits positive Lyapunov exponents almost everywhere and has absolutely continuous conditional measures along Pesin unstable leaves with respect to the corresponding Lebesgue measure. According to the absolute continuity of Pesin unstable lamination, one knows that every ergodic hyperbolic SRB measure is a physical measure.


One would like to know whether, for systems away from uniform hyperbolicity, there exist finitely many physical measures with basin covering property. This has led to many significant progresses in the study of dynamical systems. 
However, it is still far from being completely known. We refer the
readers to, e.g., \cite{BDV05,pa,v2,y02} and the references therein for more background and related conjectures on this topic.
We are concerned in this paper with the physical measures for partially hyperbolic diffeomorphisms. 

A motivation of this work is a conjecture of Palis \cite[Conjecture 2]{pa}: there is a residual or a dense set $\mathcal{R}$ of $C^r$ diffeomorphisms such that for any $f\in \mathcal{R}$ away from homoclinic tangencies, $f$ has finitely many physical measures. 
Recently, a partial answer was given by Cao-Mi-Yang \cite[Theorem A]{CMY22}: there is a dense set $\mathcal{R}$ in the set of $C^1$ diffeomorphisms such that any $f\in \mathcal{R}$ away form homoclinic tangencies, $f$ has either a physical measure supported on a sink, or an ergodic SRB measure. Due to the abundance of partially hyperbolic systems with multi one-dimensional centers for diffeomorphisms away from homoclinic tangencies \cite{CSY15}, this result was mainly concluded from the following \cite[Theorem C]{CMY22}: 
\begin{theorem}
For a $C^2$ diffeomorphism admitting the partially hyperbolic splitting with multi one-dimensonal centers, there always exists some ergodic SRB measure. 
\end{theorem}
Going back to the original conjecture \cite[Conjecture 2]{pa}, one expects to improve this result to physical measures. However, there exists examples of partially hyperbolic diffeomorphisms with one-dimensional centers without physical measures, considering the product of an Anosov diffeomorphism by the identity maps on circles. On the other hand, since any hyperbolic ergodic SRB measure is physical, to get physical measures, a natural way to proceed is to involve some hyperbolicity assumptions on the centers, see for instance \cite{ABV00,bv,AD,AV15,MCY17,HYY19}. 

In this paper, we study the physical measures for partially hyperbolic splitting systems with multi one-dimensional centers under some hyperbolicity conditions: 
\begin{itemize} 
\item we prove the finiteness of ergodic physical(SRB) measures under the condition that Gibbs $u$-states are hyperbolic; 
\item with extra assumption on the sign of central Lyapunov exponents for empirical measures, we obtain the basin covering property for these physical measures. 
\end{itemize}
The latter result is based on a criterion established for diffeomorphisms with the dominated splitting, which focusses on getting basin covering property from the finiteness of physical measures.

\subsection{Statement of main results}\label{pdd}
\medskip
Let $f$ be a $C^1$ diffeomorphism on a compact Riemannian manifold $M$. 
A $Df$-invariant splitting $TM=E\oplus F$ of the tangent bundle is called a \emph{dominated splitting} if $E$ dominates $F$, denoting by $E\oplus_{\succ} F$, in the sense that there exist $C>0$ and $0<\lambda<1$ such that 
$$
\|Df^n|_{F(x)}\|\cdot\|Df^{-n}|_{E(f^n(x))}\|\le C\lambda^n\quad \textrm{for every}~x\in M.
$$

A $C^1$ diffeomorphism $f$ is \emph{partially hyperbolic} if there exists a dominated splitting 
$TM=E^u\oplus_{\succ} E_1\oplus_{\succ}\cdots \oplus_{\succ} E_k\oplus_{\succ} E^s$ 
on the tangent bundle such that $E^u$ is uniformly expanding and $E^s$ is uniformly contracting.
By \emph{Gibbs $u$-state}, we mean an invariant measure whose conditional measures along strong unstable manifolds are absolutely continuous w.r.t. Lebesgue measures on these manifolds. We say an invariant measure is \emph{hyperbolic} if it admits only non-zero Lyapunov exponents.
%

Given $k\ge 1$, denote by $\P_k(M)$ the collection of $C^2$ diffeomorphisms in $M$ having a partially hyperbolic splitting 
$
TM=E^u\oplus_{\succ} E_1^c\oplus_{\succ}\cdots\oplus_{\succ} E_k^c \oplus_{\succ} E^s
$
such that 
\begin{itemize}
\item ${\rm dim}E_i=1, 1\le i \le k$;
\smallskip
\item every Gibbs $u$-state of $f$ is hyperbolic.
\end{itemize}

The first main result is as follows: 
%
%
%
%

\begin{theoremalph}\label{TheoA}
For every $f\in\P_k(M)$, there are finitely many ergodic physical measures of $f$.
\end{theoremalph}


As we have mentioned, \cite[Theorem C]{CMY22} have proved the existence of (ergodic) SRB measures for partially hyperbolic diffeomorphisms with one-dimensional dominated center bundles.  Observe that every SRB measure is a Gibbs $u$-state, the hyperbolicity assumption on Gibbs $u$-states guarantees the existence of ergodic physical measures. Hence, the main novelty of Theorem \ref{TheoA} is to show the finiteness of ergodic physical measures. 



Observe that when the basin covering property holds, the limit measures of the empirical measures 
$\{\frac{1}{n}\sum_{i=0}^{n-1} \delta_{f^i(x)}\}_{n\in \NN}$
for Lebesgue almost every point $x\in M$ coincide with the set of ergodic physical measures.
This implies that for Lebesgue almost every $x\in M$, the sign of Lyapunov exponents of the limit measures of empirical measures of $x$ along each center is definite. Motivated by this, we achieved the basin covering property by adding assumption on the sign of Lyapunov exponents w.r.t. empirical measures.

\begin{theoremalph}\label{TheoB}
Let $f\in\P_k(M)$.
Assume that for Lebesgue almost every point $x\in M$, there exists some limit measure of $\{\frac{1}{n}\sum_{i=0}^{n-1}\delta_{f^i(x)}\}_{n\in \NN}$ admitting the same sign(positive or negative) of Lyapunov exponents along each $E_i^c$, $2\le i \le k$. Then Lebesgue almost every point is contained in some basin of the ergodic physical measure.
\end{theoremalph}

To prove Theorem \ref{TheoB}, we will make use of the following general result. We say $\mu$ is a hyperbolic measure of index $d\in \NN$, if it is hyperbolic and admits exactly $d$ positive Lyapunov exponents for $\mu$-a.e. $x\in M$, counted with multiplicity.
%

\begin{theoremalph}\label{TheoC}
Let $f$ be a $C^2$ diffeomorphism with a dominated splitting $TM=E\oplus_{\succ} F$, which admits finitely many hyperbolic ergodic SRB measures of index ${\rm dim}E$. 

If $\K$ is a subset of $M$ with positive Lebesgue measure such that for every $x\in \K$, there exists some limit measure of $\{\frac{1}{n}\sum_{i=0}^{n-1}\delta_{f^i(x)}\}_{n\in \NN}$ to be a hyperbolic SRB measure of index ${\rm dim}E$, then Lebesgue almost every point in $\K$ belongs in the basin of some ergodic hyperbolic SRB measure of index ${\rm dim}E$.
\end{theoremalph}

Theorem \ref{TheoC} focusses on getting the basin covering property from the finiteness of ergodic physical(SRB) measures with fixed index. We mention that for diffeomorphisms in Theorem \ref{TheoC}, there may exist other ergodic SRB measures with different indexes. Thus, Theorem \ref{TheoC} provides an alternative criterion for the basin covering property in a local sense.



In the context of partially hyperbolic settings, several other results on the finiteness and basin covering property of physical measures have been obtained, especially under the hyperbolicity assumption of Gibbs $u$-states. To the best of our knowledge, we list some of them:
 
\begin{itemize}
\item The finiteness and basin covering property were also obtained by Bonatti-Viana \cite[Theorem A]{bv} for diffeomorphisms with mostly contracting center, i.e., all the Gibbs $u$-states admit only negative central Lyapunov exponents. 
\smallskip 
\item For diffeomorphisms with mostly expanding center, i.e., all the Gibbs $u$-states have only positive central Lyapunov exponents, by showing weak non-uniform expansion from mostly expanding property, and using the result of \cite{AD}, Andersson-Vasquez \cite[Theorem C]{AV15} obtained the same result on physical measures.
\smallskip
\item Inspired by the work of \cite{bv, AV15}, Mi-Cao-Yang considered in the work \cite[Theorem A]{MCY17} the case that admitting both mostly contracting center and mostly contracting center, and proved the same result on physical measures.  
\smallskip
\item Recently, a related work of Hua-Yang-Yang \cite[Corollary E]{HYY19} have shown the finiteness and basin covering property of (ergodic) physical measures for diffeomorphisms in $\P_1(M)$.
\end{itemize}

In contrast to \cite{bv,AV15,MCY17}, where the signs of central Lyapunov
exponents w.r.t. Gibbs $u$-states are definite, our hyperbolicity assumption on Gibbs $u$-states in
Theorems \ref{TheoA} and \ref{TheoB} is more flexible, allowing the coexistence of different Gibbs $u$-states with different signs of Lyapunov exponents on the same central sub-bundle. 

Note that in Theorem \ref{TheoB}, we have no assumption on the Lyapunov exponents along $E_1^c$ of empirical measures, Theorem \ref{TheoA} together with Theorem \ref{TheoB} is an extension of \cite[Corollary E]{HYY19}.
Comparing to \cite{HYY19}, influenced by the existence of several non-uniform centers, more delicate analyses on various Gibbs $cu$-states(see $\S$ \ref{ucu}) need to be involved. We wonder that if one can remove the extra assumption on the sign of Lyapunov exponents for the empirical measures to get the basin covering property.


To achieve Theorem \ref{TheoA}, we decompose the ergodic physical/SRB measures into different levels by their signs of Lyapunov exponents on centers, and then show the finiteness of these measures on each level by studying the limits of the set of ergodic Gibbs $cu$-states. 
The proof of Theorem \ref{TheoB} relies on the results of Theorems \ref{TheoA} and \ref{TheoC}. 
The main ingredient in getting basin covering property is the use of the absolute continuity of Pesin stable lamination. Due to the absence of uniformly expanding and uniformly contracting sub-bundles, the argument presented in previous works \cite{ABV00,bv,HYY19} cannot be adapted to our situation. In the proof of Theorem \ref{TheoC}, to apply the absolute continuity of Pesin stable lamination under certain uniform hyperbolicity, we renew the techniques of hyperbolic times from \cite{AP08, MC21} corresponding to points from $\K$ and their empirical measures.
%

\subsection{Organization of the paper}
The paper is organized as follows. Section \ref{pre1} contains some notions and results used in this paper. In Section \ref{S-closed}, for diffeomorphisms in $\P_k(M)$, we investigate the behavior of central Lyapunov exponents of the limit measures of the set of ergodic Gibbs $cu$-state with same index. Section \ref{A-proof} is devoted to the proof of Theorem \ref{TheoA}. 
In Section \ref{5-covering}, we reduce the proof of Theorem \ref{TheoC} to proving its weak version Theorem \ref{ext}. For a diffeomorphism $f$ satisfying assumption of Theorem \ref{TheoC}, we show some power of $f$ satisfies the assumption of Theorem \ref{ext}, with this observation we conclude the result of Theorem \ref{TheoC}. Then, by a discussion on the frequency of hyperbolic entrance times, we complete the proof of Theorem \ref{ext}.
In Section \ref{pb}, we give a proof of Theorem \ref{TheoB} by applying Theorems \ref{TheoA} and \ref{TheoC}.

\section{Preliminary}\label{pre1}
\subsection{Notations}
Throughout, we let $M$ be the compact Riemannian manifold. Denote by ${\rm Leb}$ the Lebesgue measure of $M$ given by the Riemannian metric. Given a smooth sub-manifold $D$, let ${\rm Leb}_D$ stand for the induced Lebesgue measure in $D$.
Let $\mathscr{M}$ be the space of Borel probability measures on $M$ endowed with the weak$^*$-topology. 
Given a diffeomorphism $f$ on $M$. 
For $x\in M$, denote by $\mathscr{M}(x,f)$ the set of limit measures of empirical measures
$$
\{\frac{1}{n}\sum_{i=0}^{n-1}\delta_{f^i(x)}\}_{n\in \NN}.
$$
Sometimes, we rewrite $\mathscr{M}(x,f)$ as $\mathscr{M}(x)$ when there is no ambiguity.
\subsection{Pesin blocks}\label{pdd}
%
%


Let us recall a type of Pesin blocks, whose definition is independent of invariant measures, see \cite{LLST17} for instance.
\begin{definition}\label{dpb}
Let $f$ be a $C^1$ diffeomorphism on $M$. Given $\chi_s<0<\chi_u$ and $0<\epsilon \ll \min\{\chi_u,-\chi_s\}$. The Pesin block $\Lambda_k=\Lambda_k(\chi_u,\chi_s,\epsilon)$ of level $k\in \NN$ consists of points $x\in M$ such that there is a $Df$-invariant splitting 
$$
T_xM=E^+(x)\oplus E^-(x)
$$ 
with the following properties:
\begin{itemize}
\smallskip
\item $\|Df^{-n}|_{E^+(f^m(x))}\|\le {\rm e}^{\epsilon k}{\rm e}^{(-\chi_u+\epsilon)n} {\rm e}^{\epsilon |m|}, \quad n\ge 1, m\in \ZZ$;
\smallskip
\item $\|Df^{n}|_{E^-(f^m(x))}\|\le {\rm e}^{\epsilon k}{\rm e}^{(\chi_s+\epsilon)n} {\rm e}^{\epsilon |m|}, \quad n\ge 1, m\in \ZZ$;
\smallskip
\item $|\sin \angle(E^+(f^m(x)), E^-(f^m(x)))|\ge {\rm e}^{-\epsilon k} {\rm e}^{-\epsilon |m|}, \quad m\in \ZZ$.
\end{itemize}
We say that $\Lambda=\Lambda(\chi_u,\chi_s,\epsilon)=\bigcup_{k=1}^{+\infty}\Lambda_k$ is a Pesin set.
\end{definition}


\subsection{Gibbs $u$-states and Gibbs $E$-states}\label{ucu}
Let $f$ be a $C^2$ diffeomorphism with a partially hyperbolic splitting $TM=E^u\oplus_{\succ} E^{cs}$. 
Recall that an $f$-invariant measure is a Gibbs $u$-state if its disintegration along strong unstable manifolds is absolutely continuous w.r.t. Lebesgue measures on these manifolds. Gibbs $u$-states were introduced and established by Sinai-Pesin \cite{ps82} for partially hyperbolic systems. 

We list following properties of Gibbs $u$-states, see \cite[Subsection 11.2]{BDV05} for details.

\begin{proposition}\label{uc}
Let $f$ be a $C^2$ diffeomorphism with a partially hyperbolic splitting $TM=E^u\oplus_{\succ} F$, one has following properties:
\begin{enumerate}
\item\label{11} The ergodic components of any Gibbs $u$-state are Gibbs $u$-states.
\item\label{22} The set of Gibbs $u$-states is compact in weak$^*$-topology.
\item\label{33} For Lebesgue almost every $x\in M$, any limit measure in $\mathscr{M}(x)$ 
is a Gibbs $u$-state.
\end{enumerate}
\end{proposition}

Recall the following entropy formula for diffeomorphisms with dominated splittings, see \cite[Theorem F]{CYZ18}, \cite[Theorem A]{cce15} for the proof.

\begin{proposition}\label{cug}
Let $f$ be a $C^1$ diffeomorphism with a dominated splitting $TM=E\oplus_{\succ} F$. Then for Lebesgue almost every $x\in M$, any limit measure $\mu$ in $\mathscr{M}(x)$ satisfies 
$$
{\rm h}_{\mu}(f)\ge \int \log |{\rm det}Df|_{E}| {\rm d} \mu,
$$
where ${\rm h}_{\mu}(f)$ denotes the metric entropy of $\mu$ w.r.t. $f$.
\end{proposition}

Let us recall the notion of \emph{Gibbs $cu$-states}.
Given a diffeomorphism $f$ with a dominated splitting $TM=E\oplus_{\succ} F$. An invariant measure $\mu$ called a Gibbs $cu$-state associated to $E$ if
\begin{itemize}
\item the Lyapunov exponents of $\mu$ along $E$ are positive, 
\item the conditional measures of $\mu$ along Pesin unstable manifolds tangent to $E$ are absolutely continuous w.r.t. Lebesgue measures on these manifolds.
\end{itemize}

Since we need to deal with various Gibbs $cu$-states associated to different sub-bundles, for simplicity we will just call $\mu$ a \emph{Gibbs $E$-state} if it is a Gibbs $cu$-state associated to $E$.
Note that when $E$ is uniformly expanding, then Gibbs $E$-state is indeed the Gibbs $u$-state. 

Similar to property (\ref{11}) of Proposition \ref{uc}, we have the following result. See a proof in \cite[Lemma 2.4]{cv} or \cite[Proposition 4.7]{CMY22}.
\begin{proposition}\label{cue}
Assume that $f$ is a $C^2$ diffeomorphism with a dominated splitting $TM = E\oplus_{\succ} F$. If $\mu$ is a Gibbs $E$-state, then any ergodic component of $\mu$ is also a Gibbs $E$-state.
\end{proposition}

We will use the following fact several times, which can be deduced by the absolute continuity of Pesin stable lamination, see \cite{y02} for a proof.
\begin{proposition}\label{pes}
Assume that $f$ is a $C^2$ diffeomorphism with a dominated splitting $TM = E\oplus_{\succ} F$. If $\mu$ is an ergodic Gibbs $E$-state whose Lyapunov exponents along $F$ are all negative, then $\mu$ is a physical measure. 
\end{proposition}


The result below is abstracted from \cite[Theorem 4.10]{CMY22}.
\begin{proposition}\label{ugu}
Assume that $f$ is a $C^2$ diffeomorphism with a dominated splitting $TM = E\oplus_{\succ} F$. If $\mu$ is a limit measure of the set of ergodic Gibbs $E$-states, whose Lyapunov exponents along $E$ are uniformly positive, then $\mu$ is a Gibbs $E$-state.
\end{proposition}

\section{Limit measures of the set of Gibbs $E$-states}\label{S-closed}

We will concentrate on diffeomorphisms $f$ in $\P_k(M)$ with the partially hyperbolic splitting 
$$
TM=E^u\oplus_{\succ} E_1^c\oplus_{\succ}\cdots\oplus_{\succ} E_k^c \oplus_{\succ} E^s.
$$

For each $0\le i \le k$, we write
$$
E_i=E^u\oplus_{\succ} \cdots \oplus_{\succ} E_i^c,
$$
where $E_0=E^u$ by convention. 
Let $G_{i}$ and $G_i^{erg}$ be the sets of Gibbs $E_i$-states and ergodic Gibbs $E_i$-states of $f$, respectively.
In particular, $G_0$ and $G_0^{erg}$ are the sets of Gibbs $u$-states and ergodic Gibbs $u$-states, respectively.
In this section, our main task is to study the limit measures of $G_i^{erg}$ for each fixed level $1\le i \le k$.

\smallskip
For any $f$-invariant measure $\mu$, by Oseledec's theorem \cite{bp02}, for every $1\le i \le k$, one can define
$$
\lambda^c_i(x)=\lim_{n\to +\infty}\frac{1}{n}\log \|Df^n|_{E^c_i(x)}\|, \quad \mu \textrm{-a.e.}~x.
$$
Moreover, define
$$
\lambda^c_i(\mu)=\int \log\|Df|_{E^c_i}\|{\rm d} \mu.
$$
Given an $f$-invariant measure $\mu$, denote by $\mathcal{C}(\mu)$ the family of ergodic components of $\mu$.

%

\smallskip
We recall the following folklore result \cite[Proposition 4.9]{CMY22}.
\begin{lemma}\label{sim}
Under the setting as above, we have $G_k\subset G_{k-1}\subset \cdots \subset G_0$.
\end{lemma}

The main result of this section is as follows.

\begin{theorem}\label{34}
Let $f\in \P_k(M)$ with the partially hyperbolic splitting 
$
TM=E^u\oplus_{\succ} E_1^c\oplus_{\succ}\cdots\oplus_{\succ} E_k^c \oplus_{\succ} E^s.
$
For each $0\le i \le k-1$, if $\mu$ is a limit measure of $G_i^{erg}$, then
$$
\lambda^c_{i+1}(x)> 0, \quad \mu \textrm{-a.e.}~x.
$$
Moreover, there are at most finitely many ergodic Gibbs $E_k$-states.
\end{theorem}

\subsection{Reduction of Theorem \ref{34}}
Theorem \ref{34} will be deduced from the following two theorems.

\begin{theorem}\label{sec}
Let $f$ be a diffeomorphism with a dominated splitting $TM=E\oplus_{\succ} F$. If $\mu$ is a limit measure of the set of ergodic Gibbs $E$-states, which is a Gibbs $E$-state, then for $\mu$-a.e. $x\in M$, there exist non-negative Lyapunov exponents along $F$.
\end{theorem}
%

\begin{theorem}\label{C}
If $f\in \P_k(M)$, then for every $1\le i \le k$, either $\# G_i^{erg}<+\infty$ or any limit measure of $G_i^{erg}$ is contained in $G_i$, whose Lyapunov exponents along $E_i^c$ are uniformly positive.
\end{theorem}

Theorem \ref{sec} can help us to investigate the behavior of Lyapunov exponents on complementary sub-bundle for limit measures of the set of Gibbs $E$-states. This observation will also be used in the proof of Theorem \ref{C}.
We provide the proofs of Theorems \ref{sec} and \ref{C} in next two subsections.

\medskip
Let us state the following interesting result as a byproduct of Theorem \ref{C}, it asserts the uniformity on Lyapunov exponents of Gibbs $E_i$-states along $E_i$, $1\le i \le k$. 
\begin{corollary}
For $f\in \P_k(M)$, there exists $\alpha>0$ such that for every $1\le i \le k$, for every Gibbs $E_i$-state $\mu$,
one has that 
$$
\lambda^c_i(x)>\alpha, \quad \mu \textrm{-a.e.}~x.
$$
\end{corollary}

\begin{proof}
Assuming contrary, there exist $1\le i \le k$ and a sequence of Gibbs $E_i$-states $\{\mu_n\}_{n\in \NN}$ for which there exists $\nu_n\in \mathcal{C}(\mu_n)$ for every $n\in \NN$ so that
$$
\lim_{n\to +\infty}\lambda^c_i(\nu_n)=0.
$$
%
Without loss of generality, we suppose that $\{\nu_n\}_{n\in \NN}$ converges to some $f$-invariant measure $\nu$ when $n\to +\infty$. Note that $x\mapsto \log \|Df|_{E_i^c(x)}\|$ is continuous. By the weak$^*$-convergence we have
\begin{equation}\label{eu}
\lambda^c_i(\nu)=\lim_{n\to +\infty}\lambda^c_i(\nu_n)=0.
\end{equation}
By Proposition \ref{cue}, each $\nu_n$ is an ergodic Gibbs $E_i$-state for every $n\in \NN$. Hence, $\nu$ is a limit measure of the set of ergodic Gibbs $E_i$-states. By Theorem \ref{C}, there exists $a>0$ such that 
$$
\lambda^c_i(x)>a,\quad \nu \textrm{-a.e.}~x.
$$
Noting ${\rm dim}E_i^c=1$, by applying Birkhoff's ergodic theorem (see e.g. \cite[Theorem 1.14]{wa82}) one obtains
$$
\lambda_i^c(\nu)=\int \lambda_i^c(x){\rm d}\nu>a,
$$
which is in contradiction with convergence (\ref{eu}).
\end{proof}

Now we give the proof of Theorem \ref{34} by assuming Theorems \ref{sec} and \ref{C}. 

\begin{proof}[Proof of Theorem \ref{34}]\label{se}
For $0\le i \le k-1$, let $\mu$ be a limit measure of $G_i^{erg}$.
We know from Proposition \ref{ugu} and Theorem \ref{C} that $\mu$ is a Gibbs $E_i$-state, which is also a Gibbs $u$-state by Lemma \ref{sim}. So, $\mu$ is hyperbolic by definition of $\P_k(M)$. This together with Theorem \ref{sec} ensures that for $\mu$-a.e. $x\in M$, there exist positive Lyapunov exponents along $F_i:=E_{i+1}^c\oplus_{\succ} \cdots \oplus_{\succ} E^s$. As $E_{i+1}^c$ has dimension one, $\lambda^c_{i+1}(x)$ is just the largest Lyapunov exponent along $F_i$.
Consequently, we obtain 
$$
\lambda^c_{i+1}(x)> 0, \quad \mu \textrm{-a.e.}~x.
$$ 

\smallskip
To show the result on Gibbs $E_k$-states, assume by contradiction that there are infinitely many ergodic Gibbs $E_k$-states. This implies that there exists a sequence of different ergodic Gibbs $E_k$-states, which converges to an invariant measure $\mu$. By Theorem \ref{C}, $\mu$ must be a Gibbs $E_k$-state as well. However, from Theorem \ref{sec} we get that $\mu$ have positive Lyapunov exponents along $E^s$, which is a contradiction. Thus, the proof is complete.
\end{proof}

\smallskip

\subsection{Proof of Theorem \ref{sec}}\label{sedd}
Let $\{\mu_n\}_{n\in \NN}$ be a sequence of different ergodic Gibbs $E$-states, which converges to $\mu$ as $n\to +\infty$, and $\mu$ is a Gibbs $E$-state. 

\smallskip
Arguing by contradiction, assume $A\subset M$ satisfying $\mu(A)>0$ and
$$
\lim_{n\to +\infty}\frac{1}{n}\log \|Df^n|_{F(x)}\|<0,\quad \forall~ x\in A.
$$ 
This implies that there exist ergodic components of $\mu$ whose Lyapunov exponents along $F$ are all negative. Note that these ergodic measures are Gibbs $E$-states by applying Proposition \ref{cue}. By Proposition \ref{pes}, we know that they are also physical measures, so there are at most countably many such measures. Consequently, one can find an ergodic component $\nu$ of $\mu$ such that all its Lyapunov exponents along $F$ are negative, and there exists $b\in (0,1]$ such that $\mu$ can be rewritten as
$$
\mu=b \nu+(1-b)\eta
$$
for some Gibbs $E$-state $\eta$. 

\smallskip
Now we introduce the Pesin set associated to $\nu$ as follows.
Let $\chi_u$ be the smallest positive Lyapunov exponent of $\nu$, and let $\chi_s$ be the largest negative Lyapunov exponent of $\nu$. Fix $0<\epsilon\ll \min\{\chi_u,-\chi_s\}$. One can define the Pesin set $\Lambda$ and Pesin blocks $\{\Lambda_k\}_{k\in \NN}$ associated to $(\chi_u,\chi_s,\epsilon)$ as in Definition \ref{dpb}, where we have $E^+=E$ and $E^-=F$.
Then, we know that $\nu(\Lambda)=1$ by definition. Hence, we can choose a Pesin block $\Lambda_k$ satisfying $\nu(\Lambda_{k})>0$.

\smallskip
By Pesin theory \cite[$\S2$ \& $\S4$]{bp02}, there is $\delta_k>0$ such that every $x\in \Lambda_{k}$ admits a local Pesin stable manifold $\F^s_{\delta_k}(x)$ and a local Pesin unstable manifold $\F^u_{\delta_k}(x)$ with size $\delta_k$. Moreover, from domination of the splitting one can check that $\F^u_{\delta_k}(x)$ is tangent\footnote{This means that $T_y\F^u_{\delta_k}(x)=E(y)$ for every $y\in \F^u_{\delta_k}(x)$.} to $E$ and $\F^s_{\delta_k}(x)$ is tangent to $F$.


\smallskip
Fixing $r\ll \delta_k$. By compactness of $\Lambda_{k}$, one can take $p\in \Lambda_{k}$ such that 
\begin{equation}\label{ps}
\nu(\Lambda_{k}\cap B(p,r))>0.
\end{equation}
From the continuity of dominated splitting, $r$ can be taken such that for any $y,z\in B(p,r)\cap \Lambda_k$, $\F^u_{\delta_k}(y)$ intersects $\F^s_{\delta_k}(z)$ transversely. 

\smallskip
Let us take a smooth embedded disk $D^s\ni p$ whose tangent space is close enough to $F$. Then, 
for any point $z\in \Lambda_k\cap B(p,r)$, $\F_{\delta_k}^u(z)$ intersects $D^s$ transversely at a single point, denoted as $z_s$.
Up to reducing $r$, $\F_{\delta_k}^u(z)$ would contain the disk of radius $\delta_k/2$ centered at $z_s$, denoted as $\F_{\delta_k/2}^u(z_s)$. This gives a measurable partition consists of unstable disks as follows:
$$
\cS^u=\{\mathcal{F}_{\delta_k/2}^u(z_s): z\in \Lambda_k\cap B(p,r)\}.
$$

\smallskip
Let $S^u$ be the union of unstable disks from $\cS^u$.
We have $\nu(S^u\cap \Lambda_k)>0$ by (\ref{ps}). 
The ergodicity of $\nu$ gives that 
$$
\nu(\B(\nu,f)\cap S^u\cap \Lambda_k)=\nu(S^u\cap \Lambda_k)>0.
$$
Since $\nu$ is a Gibbs $E$-state, the conditional measures of $\nu|_{S^u}$ along unstable disks of $\mathcal{S}^u$ are absolutely continuous w.r.t. Lebesgue measures on these disks, which ensures that one can take $\gamma_{\nu}\in \mathcal{S}^u$ for which there exists a subset $B\subset \gamma_{\nu}\cap \Lambda_{k}\cap \B(\nu,f)$ such that 
\begin{equation}\label{n}
{\rm Leb}_{\gamma_{\nu}}(B)>0.
\end{equation}

\smallskip
Note that for every $x\in \Lambda_{k}\cap B(p,r)$, $\F^s_{\delta_k}(x)$ intersects every local Pesin unstable manifold from $\mathcal{S}^u$ transversely. By applying the absolute continuity of Pesin stable lamination, (\ref{n}) implies
$$
{\rm Leb}_{\gamma}\left(\bigcup_{z\in B}\F_{\delta_k}^s(z)\cap \gamma\right)>0 \quad \textrm{for every}~\gamma \in  \mathcal{S}^u.
$$
Together with the observation that $\F_{\delta_k}^s(z)\subset \B(\nu,f)$ for every $z\in B$, this yields that 
\begin{equation}\label{kkf}
{\rm Leb}_{\gamma}(B(\nu,f))>0 \quad \textrm{for every}~ \gamma\in \mathcal{S}^u.
\end{equation}

\smallskip
Recall that we have $\nu(S^u)>0$, so $\mu(S^u)\ge b\nu(S^u)>0$. Up to reducing the sizes of unstable disks of $\mathcal{S}^u$ slightly, we may assume 
$
\mu(\partial(S^u))=0,
$
which implies that 
$$
\lim_{n\to +\infty}\mu_n(S^u)=\mu(S^u).
$$
In particular, there exists $n_0\in \NN$ such that $\mu_n(S^u)>0$ for every $n\ge n_0$.
For every $n\ge n_0$, since $\mu_n$ is an ergodic Gibbs $E$-state, one can find $\gamma_n\in \mathcal{S}^u$ such that ${\rm Leb}_{\gamma_n}$-almost every point of $\gamma_n$ belongs to $\B(\mu_n,f)$. Take (\ref{kkf}) into account, we see that $\mu_n=\nu$ for every $n\ge n_0$, which is contrary to our assumption.

$\hfill \qed$


\subsection{Proof of Theorem \ref{C}}\label{cp}
We prove the result by induction on $1\le i \le k$. 
Let us start by proving the case $i=1$. Assume by contradiction that $G_1^{erg}$ is infinite. Consider $\mu$ as a limit measure of $G_1^{erg}$, thus
there is a sequence of ergodic Gibbs $E_1$-states $\{\mu_n\}_{n\in \NN}$ such that 
$$
\mu_n\xrightarrow{weak^*}\mu \quad \textrm{when}~n\to +\infty.
$$

\smallskip
We have $\{\mu_n\}_{n\in \NN}\subset G_0$ as $G_1\subset G_0$ by Lemma \ref{sim}. 
From item (\ref{22}) of Proposition \ref{uc}, we know $\mu\in G_0$.
By Theorem \ref{sec} and the hyperbolicity assumption on Gibbs $u$-states, we get
$$
\lambda^c_1(x)>0,\quad \mu \textrm{-a.e.}~x.
$$

\smallskip
To complete the proof of the case $i=1$, by Proposition \ref{ugu} it suffices to verify that the Lyapunov exponents of $\mu$ are uniformly positive along $E_1^c$. Arguing by contradiction, there exists a sequence of ergodic measures $\{\nu_n\}_{n\in \NN}$ in  $\cC(\mu)$ such that 
$$
\lim_{n\to +\infty}\lambda^c_1(\nu_n)=0.
$$
%
Up to extracting a subsequence, we assume that $\{\nu_n\}_{n\in \NN}$ converges to an $f$-invariant measure $\nu$ as $n\to +\infty$. Consequently, we have
\begin{equation}\label{gud}
\lambda^c_1(\nu)=\lim_{n\to +\infty} \lambda^c_1(\nu_n)=0.
\end{equation}
Since we have shown $\mu\in G_0$, and $\{\nu_n\}_{n\in \NN}$ are ergodic components of $\mu$, item (\ref{11}) of Proposition \ref{uc} gives that $\nu_n\in G^{erg}_0$ for every $n\in \NN$. So, $\nu$ is a limit measure of $G_0^{erg}$. 
By Theorem \ref{sec} and the hyperbolicity of Gibbs $u$-states, we conclude that
$$
\lambda^c_1(x)>0,\quad \nu \textrm{-a.e.}~x.
$$ 
By Birkhoff's ergodic theorem, it follows that 
$$
\lambda^c_1(\nu)=\int \lambda^c_1(x) {\rm d}\nu>0,
$$
where we use the fact ${\rm dim}E_1^c=1$.
This gives the contradiction to (\ref{gud}). Thus, the case $i=1$ is verified.

\smallskip
Inductively, we assume that the theorem holds for $i=\ell$($\ell<k$), and we then show that it is true for $i=\ell+1$. Assume that $G_{\ell+1}^{erg}$ is an infinite set. Let $\{\mu_n\}_{n\in \NN}$ be a sequence of ergodic measures in $G_{\ell+1}$, which converges to $\mu$ as $n\to +\infty$. In view of Proposition \ref{ugu}, we need only to show
\begin{equation}\label{uni2}
\inf_{\nu\in \cC(\mu)}\left\{\lambda^c_{\ell+1}(\nu)\right\}>0.
\end{equation}

\smallskip
As Lemma \ref{sim} tells us that $G_{\ell+1}\subset G_{\ell}\subset G_0$ and we assume that the theorem is true for $i=\ell$, we get $\mu\in G_{\ell}$. 
By Theorem \ref{sec}, together with the assumption of Gibbs $u$-states we know that $\mu$ admits only positive Lyapunov exponents along $E_{\ell+1}^c$.

\smallskip
Now we begin to show (\ref{uni2}) is true. Assume by contradiction that there exists a sequence of ergodic measures $\{\nu_n\}_{n\in \NN}\subset \cC(\mu)$, which converges to an $f$-invariant measure $\nu$ and satisfies
\begin{equation}\label{et}
\lim_{n\to +\infty}\lambda^c_{\ell+1}(\nu_n)=\lambda^c_{\ell+1}(\nu)=0.
\end{equation}

Since $\mu\in G_{\ell}$,  we know from Proposition \ref{cue} that $\{\nu_n\}_{n\in\NN}$ is a sequence of ergodic measures in $G_{\ell}\subset G_0$. 
As we have assumed that this theorem is true for $i=\ell$, the limit measure $\nu$ of $\{\nu_n\}_{n\in \NN}$ is also a Gibbs $E_{\ell}$-state. By Theorem \ref{sec}, we have for $\nu$-a.e. $x\in M$ that
$
\lambda^c_{\ell+1}(x)\ge 0.
$
Moreover, it follows from item (\ref{22}) of Proposition \ref{uc} that $\nu\in G_0$. Thus, the hyperbolicity assumption on Gibbs $u$-states then gives
$$
\lambda^c_{\ell+1}(x)> 0, \quad \nu \textrm{-a.e.}~x.
$$
Therefore, using the Birkhoff's ergodic theorem again, we get
$$
\lambda^c_{\ell+1}(\nu)=\int \lambda^c_{\ell+1}(x) {\rm d}\nu>0,
$$
contradicting to the convergence (\ref{et}). 
Therefore, (\ref{uni2}) is true and we complete the proof of Theorem \ref{C}.

$\hfill \qed$

\section{Proof of Theorem \ref{TheoA}}\label{A-proof}

Let us consider $f\in \P_k(M)$ with
 the partially hyperbolic splitting 
$$
TM=E^u\oplus_{\succ} E_1^c\oplus_{\succ}\cdots\oplus_{\succ} E_k^c \oplus_{\succ} E^s.
$$
%
For every $0\le i \le k$, put 
$$
\G_i=\left\{\mu: \mu\in G_i^{erg} ,~\lambda^c_{i+1}(\mu)<0\right\}.
$$
%
Recall that $G_i^{erg}$ is the space of ergodic Gibbs $E_i$-states ($E_i=E^u\oplus_{\succ}\cdots \oplus_{\succ}E^c_i$) for every $0\le i \le k$, respectively; and 
$$
\lambda^c_{i+1}(\mu)=\int \log \|Df|_{E_{i+1}^c}\| {\rm d} \mu.
$$
Note that here we set $E_{k+1}^c=E^s$, and
$\lambda^c_{k+1}(\mu)$ denotes the maximal Lyapunov exponent of $\mu$ along $E^s$, which is negative automatically as the uniform contraction on $E^s$.

\medskip
We decompose the set of ergodic physical(SRB) measures into different levels as follows.

\begin{theorem}\label{eb}
Under the setting of Theorem \ref{TheoA}, the set of ergodic physical(SRB) measures coincides with $\bigcup_{0\le i \le k}\G_i$.
\end{theorem}

\begin{proof}
To begin with, we show the existence of ergodic physical measures.
From the result \cite[Theorem C]{CMY22}, there exists an ergodic SRB measure $\mu$ for $f$, which is a Gibbs $u$-state by definition. We know also that $\mu$ is hyperbolic by assumption. Consequently, $\mu$ is an ergodic physical measure by using the absolute continuity of Pesin stable lamination. 

\smallskip
By definition, for any $0\le i \le k$, any invariant measure in $\G_i$ is a hyperbolic ergodic SRB measure, which is also a physical measure by Proposition \ref{pes}. Thus, it remains to show that every ergodic physical measure is contained in $\G_i$ for some $0\le i \le k$.

\smallskip
From item (\ref{33}) of Proposition \ref{uc} together with Proposition \ref{cug}, one can take a subset $\D$ of $M$ with full Lebesgue measure such that for any $x\in \D$, any limit measure $\eta\in\mathscr{M}(x)$ is a Gibbs $u$-state and satisfies
\begin{equation}\label{ee}
{\rm h}_{\eta}(f)\ge \int \log |{\rm det}Df|_{E_{i}}| {\rm d} \eta,\quad \forall\, 1\le i \le k.
\end{equation}

\smallskip
Let $\mu$ be an ergodic physical measure of $f$. It follows from the definitions of physical measure and $\D$ that ${\rm Leb}(\B(\mu,f)\cap \D)>0$. By taking $x\in \B(\mu,f)\cap \D$, one gets from definition of $\B(\mu,f)$ that
$$
\lim_{n\to +\infty}\frac{1}{n}\sum_{j=0}^{n-1}\delta_{f^j(x)}=\mu.
$$
Hence, $\mu$ is a Gibbs $u$-state and satisfies (\ref{ee}) form the definition of $\D$.

\smallskip
Since any Gibbs $u$-state is hyperbolic by assumption, $\mu$ is an ergodic hyperbolic Gibbs $u$-state. Consequently, one can take $i\in \{0,\cdots,k\}$ such that all the Lyapunov exponents of $\mu$ along $E_i$ are positive and $\lambda^c_{i+1}(\mu)<0$. By Ruelle's inequality \cite{ru78}, one has
$$
{\rm h}_{\mu}(f)\le \int \log |{\rm det}Df|_{E_{i}}| {\rm d} \mu.
$$
On the other hand, (\ref{ee}) tells us that
$$
{\rm h}_{\mu}(f)\ge \int \log |{\rm det}Df|_{E_{i}}| {\rm d} \mu.
$$
Therefore, 
$$
{\rm h}_{\mu}(f)= \int \log |{\rm det}Df|_{E_{i}}| {\rm d}\mu.
$$
By the result of Ledrappier-Young \cite[Theorem A]{ly}, $\mu$ is an SRB measure, which is contained in $\G_i$ by its sign of Lyapunov exponents.
\end{proof}

Now, we can give the proof of Theorem \ref{TheoA}.

\begin{proof}[Proof of Theorem \ref{TheoA}]

In view of Theorem \ref{eb}, it suffices to show that for each $0\le i \le k$, $\G_i$ is a finite set. We know from Theorem \ref{34} that $\G_k$ is finite. Now we assume 
$0\le i \le k-1$. 

\smallskip
By contradiction, there exists a sequence of measures $\{\mu_n\}_{n\in \NN}\subset \G_i$, which converges to some invariant measure $\mu$ as $n\to +\infty$.
According to Theorem \ref{34}, $\mu$ satisfies 
\begin{equation}\label{t}
\lambda^c_{i+1}(x)>0, \quad \mu \textrm{-a.e.}~ x\in M.
\end{equation}

%
%

\smallskip
On the other hand, since $\{\mu_n\}_{n\in \NN}\subset \G_i$, we know by definition that $\lambda^c_{i+1}(\mu_n)<0$ for every $n\in\NN$. The convergence $\mu_n\xrightarrow{weak^*} \mu$ yields
\begin{equation}\label{j}
\lambda^c_{i+1}(\mu)=\lim_{n\to +\infty}\lambda^c_{i+1}(\mu_n)\le 0.
\end{equation}
By Birkhoff's ergodic theorem, we also have
$$
\lambda^c_{i+1}(\mu)=\int \lambda^c_{i+1}(x) {\rm d}\mu.
$$
This together with (\ref{j}) yields that there exists a subset $A\subset M$ such that $\mu(A)>0$ and
$$
\lambda^c_{i+1}(x)<0, \quad \forall\, x\in A,
$$
which contradicts (\ref{t}) and we complete the proof.
\end{proof}

\section{From finiteness to covering}\label{5-covering}

Let $f$ be a $C^2$ diffeomorphism with the dominated splitting 
$
TM=E\oplus_{\succ} F.
$ 
Denote by $\cS(f)$ and $\cS_{erg}(f)$ the sets of hyperbolic SRB measures and hyperbolic ergodic SRB measures of index ${\rm dim}E$, respectively. We know every element of $\cS_{erg}(f)$ is physical by applying the absolute continuity of Pesin stable lamination. 

\smallskip
Here we state a result on the basin covering property of physical measures from $\cS_{erg}(f)$ under the finiteness condition.   

%

\begin{theorem}\label{ext}
Let $f$ be a $C^2$ diffeomorphism with a dominated splitting 
$
TM=E\oplus_{\succ} F.
$
Assume that $\#\cS_{erg}(f)<+\infty$, and 
\begin{equation}\label{ss}
\int \log \|Df^{-1}|_{E}\|{\rm d}\nu<0, \quad \forall\, \nu\in \cS_{erg}(f).
\end{equation}

If $\K$ is a subset of $M$ with positive Lebesgue measure such that for every $x\in \K$, there exists $\mu\in \mathscr{M}(x)$ that belongs to $\cS(f)$, then Lebesgue almost every point in $\K$ belongs in the basin of some $\nu\in \cS_{erg}(f)$.
\end{theorem}


\smallskip
Let us remark that Theorem \ref{ext} is a variation of Theorem \ref{TheoC}, in contrast to Theorem \ref{TheoC}, we have the further assumption (\ref{ss}) on Lyapunov exponents along $E$ in the present theorem. Nevertheless, one can deduce Theorem \ref{TheoC} from Theorem \ref{ext}, which is done in Subsection \ref{tc}. The main goal of the left two subsections is to prove Theorem \ref{ext}.


\subsection{Proof of Theorem \ref{TheoC}}\label{tc}

Our strategy of the proof of Theorem \ref{TheoC} is to apply Theorem \ref{ext} to some power of the diffeomorphism considered in Theorem \ref{TheoC}.

\smallskip
The following result tells us that property (\ref{ss}) can be guaranteed for some power of $f$ by finiteness of  $\cS_{erg}(f)$.

\begin{proposition}\label{edd}
Let $f$ be a $C^2$ diffeomorphism with the dominated splitting $TM=E\oplus_{\succ}F$.
If $\#\cS_{erg}(f)<+\infty$, then there exist $N\in \NN$ and $\alpha<0$ such that 
\begin{equation}\label{me}
\int \log \|Df^{-N}|_E\| {\rm d}\mu<\alpha, \quad \forall \,\mu \in \cS(f^N).
\end{equation}
\end{proposition}

\medskip
To prove Proposition \ref{edd}, we recall the following result. One can see \cite[Lemma 3.5]{MCY17} for a detailed proof.

\begin{lemma}\label{N-iterate}
Let $f$ be a $C^2$ diffeomorphism with a dominated splitting $TM=E\oplus_{\succ}F$.
Given $L\in\mathbb N$, $\alpha_1<0$ and $\alpha_2\in (N\alpha_1/L,0)$, there exists $N$ as a multiple of $L$ such that for any $f$-ergodic measure $\mu$, if $\int \log \|Df^{-L}|_{E}\|{\rm d}\mu<\alpha_1$, then
$$
\lim_{n\to\infty}\frac{1}{n}\sum_{i=1}^{n}\log \|Df^{-N}|_{E(f^{i N}(x))}\|<\alpha_2,\quad \mu \textrm{-a.e.}~x\in M.
$$
\end{lemma}

\medskip
Now we give the proof of Proposition \ref{edd}.

\begin{proof}[Proof of Proposition \ref{edd}]
Let $\cS_{erg}(f)=\{\nu_i: 1\le i\le k\}$, as $\#\cS_{erg}(f)<+\infty$ by assumption. For each $\nu_i\in \cS_{erg}(f)$, since its Lyapunov exponents along $E$ are all positive, there is $\alpha_i<0$ satisfying
$$
\lim_{n\to +\infty}\frac{1}{n}\log \|Df^{-n}|_{E(x)}\|<\alpha_i,\quad \nu_i\textrm{-a.e.}~x\in M. 
$$
Integrating with respect to $\nu_i$, and then by applying dominated convergence theorem we obtain that there exists $L_i\in \NN$ such that
\begin{equation}\label{ii}
\int \log \|Df^{-n}|_E\| {\rm d}\nu_i<\frac{\alpha_i n}{2}, \quad \forall\, n\ge L_i.
\end{equation}
By taking
$$
L_0=\max\{L_i: 1\le i \le k\},\quad \alpha_0=\max\left\{\frac{\alpha_i L_0}{2}: 1\le i \le k\right\},
$$
we get from (\ref{ii}) that 
$$
\int \log \|Df^{-L_0}|_E\|{\rm d}\nu_i<\alpha_0, \quad \forall \,1\le i \le k.
$$
Note that each $\nu_i$, $1\le i \le k$ is ergodic for $f$, according to Lemma \ref{N-iterate}, there exist $N\in \NN$ and $\alpha<0$ such that for every $1\le i \le k$, the set 
$$
\mathcal{L}_i=\left\{x\in M: \lim_{n\to +\infty}\frac{1}{n}\sum_{j=1}^{n}\log \|Df^{-N}|_{E(f^{jN}(x))}\|<\alpha\right\}
$$
admits the full $\nu_i$-measure. 


\medskip
Now we begin to complete the proof for above $N$ and $\alpha$. By ergodic decomposition theorem, it suffices to verify the conclusion for measures in $\cS_{erg}(f^N)$. For each $\mu \in \cS_{erg}(f^N)$, consider 
$$
\widehat{\mu}=\frac{1}{N}\sum_{i=0}^{N-1}f_{\ast}^i\,\mu.
$$
Then, we have $\widehat{\mu}\in \cS_{erg}(f)$ by definition. As a consequence, $\widehat{\mu}$ must coincide with some $\nu_i$, $1\le i \le k$, which implies that $\widehat{\mu}(\mathcal{L}_i)=1$. So, we obtain $\mu(\mathcal{L}_i)=1$ by construction of $\widehat{\mu}$.
Since $\mu$ is ergodic for $f^N$, the Birkhoff's ergodic theorem says that $\B(\mu,f^N)$ has full $\mu$-measure as well. 
It follows that $\mu(\mathcal{L}_i\cap \B(\mu,f^N))=1$. We conclude from the definitions that
$$
\int \log \|Df^{-N}|_{E}\|{\rm d}\mu=\lim_{n\to +\infty}\frac{1}{n}\sum_{j=1}^n\log \|Df^{-N}|_{E(f^{jN}(x))}\|<\alpha
$$
holds for every $x\in \mathcal{L}_i\cap \B(\mu,f^N)$. This completes the proof of Proposition \ref{edd}.
\end{proof}

\medskip
The following interesting result will be used later, which can be deduced directly from the definition, thus we omit the proof. 

\begin{lemma}\label{fnf}
Let $f$ be a $C^2$ diffeomorphism with the dominated splitting $TM=E\oplus_{\succ}F$. Then for every $N\in \NN$, we have that $\cS_{erg}(f)$ is finite if and only if $\cS_{erg}(f^N)$ is finite. Moreover, we have 
$$
\#\cS_{erg}(f)\le \#\cS_{erg}(f^N)\le N\#\cS_{erg}(f).
$$
\end{lemma}

We need the next observation in the proof of Theorem \ref{ext}.

\begin{lemma}\label{lemx}
Let $f$ be a homeomorphism on $M$. Given $N\in \NN$ and $x\in M$. For every $\mu\in \mathscr{M}(x,f)$, there exists $\widehat{\mu}\in \cM(x,f^N)$ such that 
$$
\mu=\frac{1}{N}\sum_{\ell=0}^{N-1}f_{\ast}^{\ell} \,\widehat{\mu}.
$$
\end{lemma}

\begin{proof}
By definition, there exists a subsequence $\{n_k\}_{k\in \NN}$ such that
\begin{equation}\label{sde}
\frac{1}{n_k}\sum_{i=0}^{n_k-1}\delta_{f^i(x)}\xrightarrow{weak^*} \mu \quad {\rm as}~k\to +\infty.
\end{equation}
For any sufficiently large $k$, we write 
$n_k=j_kN+t$ for some $j_k\in \NN$ and $0\le t \le N-1$.
Now we show 
\begin{equation}\label{kkd}
\frac{1}{j_kN}\sum_{i=0}^{j_kN-1}\delta_{f^i(x)}\xrightarrow{weak^*} \mu \quad {\rm as}~k\to +\infty.
\end{equation}
According to (\ref{sde}), it suffices to show that for any $\varphi\in C(M)$ one has
\begin{equation}\label{one}
\lim_{k\to +\infty}\frac{1}{n_k}\sum_{i=0}^{n_k-1}\varphi(f^i(x))=\lim_{k\to +\infty}\frac{1}{j_kN}\sum_{i=0}^{j_kN-1}\varphi(f^i(x)).
\end{equation}

\smallskip
Observe first that
\begin{equation}\label{fir}
\frac{1}{n_k}\sum_{i=0}^{n_k-1}\varphi(f^i(x))=\frac{j_kN}{n_k}\cdot \frac{1}{j_kN}\sum_{i=0}^{j_kN-1}\varphi(f^i(x))+\frac{1}{n_k}\sum_{i=j_kN}^{n_k-1}\varphi(f^i(x)).
\end{equation}
Since we have
$$
\left|\frac{1}{n_k}\sum_{i=j_kN}^{n_k-1}\varphi(f^i(x))\right|\le \frac{1}{n_k}\sum_{i=j_kN}^{n_k-1}\left|\varphi(f^i(x))\right|\le \frac{t}{n_k}\|\varphi\|\le \frac{N-1}{n_k}\|\varphi\|,
$$
where $\|\varphi\|=\max_{z\in M}|\varphi(z)|$, we see that 
\begin{equation}\label{two}
\lim_{k\to +\infty}\frac{1}{n_k}\sum_{i=j_kN}^{n_k-1}\varphi(f^i(x))=0.
\end{equation}
From $n_k-N+1\le j_kN \le n_k$, we know 
\begin{equation}\label{jj}
\lim_{k\to +\infty}\frac{j_kN}{n_k}=1.
\end{equation}
Combining (\ref{fir}), (\ref{two}) and (\ref{jj}), one concludes the desired result (\ref{one}).


\smallskip
Now we are ready to find component $\widehat{\mu}$ from $\mu$.
For each $k\in \NN$, put
$$
\nu_{\ell}^k(x)=\frac{1}{j_k}\sum_{m=0}^{j_k-1}\delta_{f^{\ell+m N}(x)}, \quad 0\le \ell \le N-1.
$$
By definition, we have the presentation
$$
\frac{1}{N}\sum_{\ell=0}^{N-1}\nu_{\ell}^k(x)=\frac{1}{j_kN}\sum_{i=0}^{j_kN-1}\delta_{f^i(x)}.
$$
Up to considering a subsequence, $\{\nu_0^k(x)\}_{k\in \NN}$ converges to some $\widehat{\mu}$ as $k\to +\infty$, which means that
$$
\frac{1}{j_k}\sum_{m=0}^{j_k-1}\delta_{f^{m N}(x)}\xrightarrow{weak^*} \widehat{\mu} \quad {\rm as}~k\to +\infty.
$$
Thus, we have $\widehat{\mu}\in \cM(x,f^N)$ from definition. Moreover, this also implies that for every $0\le \ell \le N-1$, one has that
$$
\lim_{k\to +\infty}\nu_{\ell}^k(x)=\lim_{k\to +\infty}f_{\ast}^{\ell}\,\left(\nu_0^k(x)\right)=f_{\ast}^{\ell}\,\widehat{\mu}.
$$
By convergence (\ref{kkd}), we conclude that
$$
\mu=\frac{1}{N}\sum_{\ell=0}^{N-1}f_{\ast}^{\ell}\, \widehat{\mu}.
$$
This completes the proof. 
\end{proof}

\medskip
Now we are ready to prove Theorem \ref{TheoC} by admitting Theorem \ref{ext}

\begin{proof}[Proof of Theorem \ref{TheoC}]
Let $f$ be a $C^2$ diffeomorphism given by Theorem \ref{TheoC}.
We will first show that there exists some $N\in \NN$ such that $f^N$ satisfies the hypothesis of Theorem \ref{ext}. More precisely, we need to show 
\begin{itemize}
\item for every $\nu\in \cS_{erg}(f^N)$, we have
$$
\int \log \|Df^{-N}|_E\| {\rm d}\nu<0;
$$
\item $\#\cS_{erg}(f^N)<+\infty$;
\smallskip
\item for each $x\in \K$, one can find $\widehat{\mu}\in \cS(f^N)$ from $\mathscr{M}(x,f^N)$.
\end{itemize}

\smallskip
Since $\#\cS_{erg}(f)<+\infty$ by assumption, according to Proposition \ref{edd}, there exist $N\in \NN$ and $\alpha<0$ such that 
$$
\int \log \|Df^{-N}|_E\| {\rm d}\nu<\alpha, \quad \forall\,\nu \in \cS_{erg}(f^N).
$$
Thus, the first assertion is achieved for this $N$. 
We then get the second assertion by Lemma \ref{fnf}. Hence, it remains to show the last assertion. To this end, we take any $x\in \K$, by assumption of Theorem \ref{TheoC}, one can find $\mu\in \cS(f)$ in $\mathscr{M}(x,f)$. By Lemma \ref{lemx}, there exists $\widehat{\mu}\in \cM(x,f^N)$ satisfying
$$
\mu=\frac{1}{N}\sum_{\ell=0}^{N-1}f_{\ast}^{\ell}\, \widehat{\mu}.
$$
We have $\widehat{\mu}\in \cS(f^N)$ by definition.
Altogether, we have shown that $f^N$ satisfies the assumption of Theorem \ref{ext}. 

\smallskip
Assume that $\cS_{erg}(f^N)=\{\widehat{\nu}_1,\cdots,\widehat{\nu}_p\}$, as it is finite.  
By Theorem \ref{ext}, we know that Lebesgue almost every point of $\K$ is contained in $\cup_{1\le i \le p}\B(\widehat{\nu}_i,f^N)$. 
Taking
$$
\nu_i=\frac{1}{N}\sum_{\ell=0}^{N-1}f_{\ast}^{\ell}\,\widehat{\nu}_i, \quad 1\le i \le p.
$$
One can check directly that $\nu_1,\cdots,\nu_p$ are elements\footnote{These measures may be counted more than once.} of $\cS_{erg}(f)$.
Observe also that $\B(\widehat{\nu}_i,f^N)\subset \B(\nu_i,f)$ by definition, which implies that Lebesgue almost every point of $\K$ is contained in $\cup_{1\le i \le p}\B(\nu_i,f)$.

\smallskip
The proof of Theorem \ref{TheoC} is complete now.
\end{proof}

\subsection{Hyperbolic entrance times}\label{ht}
Throughout this subsection, we assume that $f$ is a $C^2$ diffeomorphism with the dominated splitting 
$
TM=E\oplus_{\succ} F.
$ 

\medskip
Given $\sigma<0$ and $x\in M$, denote by $n\in \H(x,\sigma)$ if
$$
\frac{1}{k}\sum_{n-k+1}^{n}\log\|Df^{-1}|_{E(f^i(x))}\|\le \sigma,\quad \forall\, 1\le k \le n.
$$
Sometimes, we call $n$ a $\sigma$-\emph{hyperbolic time} whenever $n\in \H(x,\sigma)$. 
%
%
%
Given $U\subset M$, define the set of $\sigma$-\emph{hyperbolic entrance times} of $(x,U)$ as follows:
$$
\mathcal{H}(x,\sigma,U)=\left\{k\in \H(x,\sigma):~ f^k(x)\in U\right\}.
$$

\medskip
Given $\JJ\subset \NN$, for every $n\in \NN$ we define $\JJ|_{n}=\JJ\cap\{1,\cdots,n\}$. Let us recall the following lemma, which is a direct consequence of Pliss Lemma \cite[Lemma 3.1]{ABV00}.
\begin{lemma}\label{kp}
Given $n\in \NN$ and $\alpha<0$, there exists $\theta:=\theta(\alpha,f)\in (0,1)$ such that  
$$
\frac{1}{n}\sum_{i=1}^n\log \|Df^{-1}|_{E(f^i(x))}\|<\alpha,
$$
then 
$$
\frac{1}{n}\#\{\H(x,\alpha/2)|_n\}\ge \theta.
$$
\end{lemma}

\medskip
The following result provides a condition on getting infinitely many hyperbolic entrance times.

\begin{proposition}\label{pro2}
Given $\alpha<0$, there exist $\sigma<0$, $0<\beta<1$ such that for every $x\in M$, if there exist $\mu\in \mathscr{M}(x)$ and an open subset $U$ satisfying 
$$
\int \log \|Df^{-1}|_{E}\|{\rm d}\mu<\alpha, \quad \mu(U)>\beta,
$$
then 
$$
\limsup_{n\to +\infty}\frac{1}{n}\#\{\mathcal{H}(x,\sigma,U)|_n\}>0.
$$
\end{proposition}


\begin{proof}
Let $x$ be a point of $M$ for which there exists $\mu\in \cM(x)$ such that
\begin{equation}\label{ppg}
\int \log \|Df^{-1}|_{E}\|{\rm d}\mu<\alpha.
\end{equation}
By definition of $\cM(x)$, one can find subsequence $\{n_k\}_{k\in \NN}$ satisfying 
\begin{equation}\label{nk}
\frac{1}{n_k}\sum_{i=1}^{n_k}\delta_{f^i(x)}\xrightarrow{weak*}\mu\quad {\rm when}~k\to +\infty.
\end{equation}
This together with (\ref{ppg}) gives
$$
\lim_{k\to +\infty}\frac{1}{n_k}\sum_{i=1}^{n_k}\log \|Df^{-1}|_{E(f^i(x))}\|=\int \log \|Df^{-1}|_E\| {\rm d}\mu<\alpha.
$$
By applying Lemma \ref{kp} to integers $n_k$, $k\in \NN$, there exists $\theta\in (0,1)$ so that
\begin{equation}\label{rt}
\liminf_{k\to +\infty}\frac{1}{n_k}\#\{\H(x,\alpha/2)|_{n_k}\}>\theta.
\end{equation}
Let $\beta\in (1-\theta, 1)$ and $U$ be any open subset satisfying $\mu(U)>\beta$. It follows from (\ref{nk}) that
$$
\liminf_{k\to +\infty}\frac{1}{n_k}\sum_{i=1}^{n_k}\chi_{U}(f^i(x))\ge \mu(U)>\beta.
$$
Combining this with (\ref{rt}), one concludes by the choice of $\beta$ that 
$$
\liminf_{k\to +\infty}\frac{1}{n_k}\#\{\mathcal{H}(x,U, \alpha/2)|_{n_k}\}>\theta+\beta-1>0.
$$
This clearly implies the desired result.
\end{proof}

As an application of Proposition \ref{pro2}, we have the next result.

\begin{corollary}\label{pro1}
Under the assumptions of Theorem \ref{ext}, there exist $\sigma<0$, $0<\beta<1$ such that for every $x\in \K$, if $\mu\in \mathscr{M}(x)$ belongs to $\mathcal{S}(f)$, and $U$ is an open subset satisfying $\mu(U)>\beta$, then 
$$
\limsup_{n\to +\infty}\frac{1}{n}\#\{\mathcal{H}(x,\sigma,U)|_n\}>0.
$$
\end{corollary}

\begin{proof}
Let $\cS_{erg}(f)=\{\nu_i: 1\le i \le \ell\}$ for some $\ell\in \NN$, recalling $\#\cS_{erg}(f)<+\infty$ in Theorem \ref{ext}. Moreover, by assumption (\ref{ss}) one can take a uniform constant $\alpha<0$ such that 
\begin{equation}\label{rr}
\int \log \|Df^{-1}|_{E}\|{\rm d}\nu_i<\alpha, \quad \forall \, 1\le i \le \ell.
\end{equation}
For any $x\in \K$, consider $\mu\in \mathscr{M}(x)$ that belongs to $\cS(f)$. By property (\ref{11}) of Proposition \ref{uc}, the ergodic components of $\mu$ are contained in $\cS_{erg}(f)$. Then, by applying ergodic decomposition theorem, there exist $c_i\in [0,1]$, $1\le i\le \ell$ such that 
$$
\mu=\sum_{i=1}^{\ell}c_i\nu_i, \quad \sum_{i=1}^{\ell}c_i=1.
$$
Take (\ref{rr}) into account, we get
$$
\int \log \|Df^{-1}|_{E}\|{\rm d}\mu=\sum_{i=1}^{\ell}c_i\int \log \|Df^{-1}|_{E}\|{\rm d}\nu_i<\alpha.
$$
One then concludes by Proposition \ref{pro2}. 
\end{proof}

\subsection{Proof of Theorem \ref{ext}}\label{ttc}
Given $a>0$, for every $x\in M$, define the $E$-direction \emph{cone of width} $a$ at $x$ as follows:
$$
\mathscr{C}^E_{a}(x)=\left\{v=v^E\oplus v^F\in E(x)\oplus F(x): \|v_F\|\le a \|v_E\|\right\}.
$$
We say a $C^1$ embedded sub-manifold $D$ is \emph{tangent to} $\mathscr{C}^E_{a}$ if it has dimension ${\rm dim}E$ and $T_xD\subset \mathscr{C}^E_{a}(x)$ for every $x\in D$.

\medskip
Let us recall the following result, which is established in \cite[Lemma 5.4]{AP08} by applying the backward contracting property on hyperbolic times.

\begin{proposition}\label{gtg}
Given $\sigma<0$, $a>0$ there exists $\xi>0$ such that for any $C^2$ disk $D$ transverse to $F$, if $H$ and $U$ are subsets of $M$ such that
\begin{itemize}
\item ${\rm Leb}_D(H)>0$;
\smallskip
\item $\mathcal{H}(x,\sigma,U)$ is infinite for every $x\in H$.
\end{itemize}
Then, for any $\zeta\in (0,1)$, there exists $n\in \mathcal{H}(x,\sigma,U)$ with $x\in H\cap D$ such that $f^n(D)$ is tangent to $\mathscr{C}^{E}_a$ and contains the ball $B_{\xi}(f^n(x))$ of radius $\xi$ centered at $f^n(x)$, which satisfies
$$
\frac{{\rm Leb}_{f^n(D)}\left(f^n(H)\cap B_{\xi}(f^n(x))\right)}{{\rm Leb}_{f^n(D)}\left(B_{\xi}(f^n(x))\right)}>\zeta.
$$
\end{proposition}

\medskip

We now fix a $C^2$ diffeomorphism $f$ as in Theorem \ref{ext}.
In order to prove Theorem \ref{ext}, we will apply Corollary \ref{pro1} and Proposition \ref{gtg} to some special open subset $U$, relating to Pesin blocks. 

\smallskip
Let us start by introducing Pesin blocks in terms of measures of $\cS_{erg}(f)$.
As $\#\cS_{erg}(f)<+\infty$, we write $\cS_{erg}(f)=\{\nu_i: 1\le i \le \ell\}$ from now on. For each $1\le i \le \ell$, denote by $\chi_i$ and $\psi_i$ the smallest Lyapunov exponents of $\nu_i$ along $E$ and the largest Lyapunov exponents of $\nu_i$ along $F$, respectively.
Let us take 
$$
\chi_u=\min\{\chi_i: 1\le i \le \ell\} \quad \textrm{and}\quad \chi_s=\max\{\psi_i: 1\le i \le \ell\}.
$$
Fix $\epsilon\ll \min\{\chi_u, -\chi_s\}$ and consider Pesin blocks $\Lambda_k=\Lambda_k(\chi_u,\chi_s,\epsilon)$, $k\in \NN$ defined as in Definition \ref{dpb}.


%
We construct some special open neighborhood of the Pesin block as follows.

\begin{lemma}\label{ball}
Given $\beta>0$ and $k\in \NN$ such that $\nu(\Lambda_k)>\beta$ for every $\nu\in \cS_{erg}(f)$. Then for every $r>0$, there exists a family $\F(r)$ consisting of $r$-balls such that
\begin{itemize}
\item for each $B\in \F(r)$, there exists $\nu\in\cS_{erg}(f)$ such that $\nu(B\cap \Lambda_k)>0$;
\smallskip
\item one has that
$$
\mu\left(\bigcup_{B\in \F(r)}B\right)>\beta, \quad \forall \,\mu \in \cS(f).
$$
\end{itemize} 
\end{lemma}

\begin{proof}
Recall that $\cS_{erg}(f)=\{\nu_i: 1\le i \le \ell\}$. For each $1\le i \le \ell$, since ${\rm supp}(\nu_i|\Lambda_k)$ is compact, there exists an open covering $\F^i(r)$ of ${\rm supp}(\nu_i|\Lambda_k)$ with some cardinality $n_i\in \NN$ defined as follows 
$$
\F^i(r)=\left\{B(x_{i_j},r): x_{i_j}\in {\rm supp}(\nu_i|\Lambda_k),~ 1\le j\le n_i\right\}.
$$
This implies that we have $\nu_i(B\cap \Lambda_k)>0$ for every $B\in \F^i(r)$. Moreover, as $\nu_i(\Lambda_k)>\beta$ by assumption, it follows that
\begin{equation}\label{kk}
\nu_i\left(\bigcup_{B\in \F^i(r)}B\right)\ge \nu_i(\Lambda_k)>\beta.
\end{equation}
Take
$$
\F(r)=\bigcup_{i=1}^{\ell}\F^i(r).
$$
The first item can be deduced  by construction immediately. To show the second item, let us fix any $\mu\in \cS(f)$. By ergodic decomposition theorem, there exists $c_i\in [0,1]$ for every $1\le i \le \ell$ such that 
$$
\mu=\sum_{i=1}^{\ell}c_i\nu_i, \quad \sum_{i=1}^{\ell}c_i=1.
$$
Thus, in view of (\ref{kk}) one has
\begin{eqnarray*}
\mu\left(\bigcup_{B\in \F(r)}B\right) &=& \mu\left(\bigcup_{i=1}^\ell\bigcup_{B\in \F^i(r)}B\right)\\
&\ge & \sum_{i=1}^{\ell} c_i\nu_i\left(\bigcup_{B\in \F^i(r)}B\right)\\
&>& \beta,
\end{eqnarray*}
which shows the second item, thus the proof is complete.

\end{proof}

\medskip
Now we can give the proof of Theorem \ref{ext}.

\begin{proof}[Proof of Theorem \ref{ext}]

Let $\cS_{erg}(f)=\{\nu_i: 1\le i \le \ell\}$.
Let $\sigma<0$, $0<\beta<1$ be the constants given by Corollary \ref{pro1}. Since the sequence of Pesin blocks $\{\Lambda_{k}\}_{k\in \NN}$ is increasing, one can take $k$ large enough so that 
$$
\nu_i(\Lambda_{k})>\beta, \quad \forall\, 1\le i\le \ell.
$$
Take $\delta_k$ as the size of local Pesin (un)stable manifolds of points in $\Lambda_k$. Fix the cone-width $a$ sufficiently small, let $\xi$ be the constant given by Proposition \ref{gtg}.
We fix $\delta=\min\{\xi,\delta_{k}\}$ now.

\smallskip
By domination, one can choose $r\ll \delta$ such that, for every $B(x,r)$, for every $y,z\in B(x,r)$, if $D_y$ and $D_z$ are $C^1$ disks of radius $\delta$ that tangent to $\mathscr{C}_a^E$ and $\mathscr{C}_a^F$, respectively, then $D_y$ intersects $D_z$ transversely.
For every $x\in \Lambda_k$, denote by $\F^s_{\delta}(x)$ and $\F^u_{\delta}(x)$ the local Pesin stable manifold and local Pesin unstable manifold at $x$ of size $\delta$, respectively. As we have noted in the proof of Theorem \ref{sec}, $\F^s_{\delta}(x)$ is tangent to $E$ and $\F^u_{\delta}(x)$ is tangent to $F$. In particular, they are tangent to $\mathscr{C}_a^E$ and $\mathscr{C}_a^F$, respectively.

\smallskip
Applying Lemma \ref{ball} to $\beta$ and $r$ as above, we get the family $\F(r)$ consisting of finitely many $r$-balls with the following properties:
\begin{itemize}
\item for each $B\in \F(r)$, there exists $\nu_i\in\cS_{erg}(f)$ such that $\nu_i(B\cap \Lambda_k)>0$;
\item 
\begin{equation}\label{kkr}
\mu\left(\bigcup_{B\in \F(r)}B\right)>\beta, \quad \forall\, \mu \in \cS(f).
\end{equation}
\end{itemize}

\smallskip
For each $B\in \F(r)$, as in the proof of Theorem \ref{sec},
we can build a measurable partition $\cS^u_B$ surrounding $B$, whose elements are local unstable disks of size $\delta/2$ of points from $\Lambda_k\cap B$.


\smallskip
We have the following claim, whose proof we postpone for a while:
\begin{claim}
There exists $\eta>0$ such that for each $1\le i \le \ell$, for each $B\in \F^i(r)$, if $x\in B$ and 
$\gamma$ is a disk of radius $\xi$ centered at $x$, then 
\begin{equation}\label{ttg}
\frac{{\rm Leb}_{\gamma}(\B(\nu_i,f))}{ {\rm Leb}_{\gamma}(\gamma)}>\eta.
\end{equation}
\end{claim}

\medskip
Let
$$
H:=\K\setminus \bigcup_{i=1}^{\ell}\B(\nu_i,f),\quad U:=\bigcup_{B\in \F(r)}B.
$$
Going by contradiction, we assume that
$
{\rm Leb}(H)>0.
$
Let us take an open ball $\mathcal{O}$ such that ${\rm Leb}(\mathcal{O}\cap H)>0$, and foliate $\mathcal{O}$ with smooth disks transverse to $F$. By applying the Fubini's  theorem, we may fix a $C^2$ disk $D$ transverse to $F$ such that ${\rm Leb}_D(H)>0$.

\smallskip
By definition of $\K$, for any $x\in H\subset \K$, any $\mu\in \mathscr{M}(x)$ is contained in $\cS(f)$. Therefore, we have
$$
\mu(U)=\mu\left(\bigcup_{B\in \F(r)}B\right)>\beta
$$
from (\ref{kkr}). By applying Corollary \ref{pro1} to above $U$, we have
$$
\limsup_{n\to +\infty}\frac{1}{n}\#\{\mathcal{H}(x,\sigma,U)|_n\}>0, \quad \forall\, x\in H.
$$
By Proposition \ref{gtg}, for fixed $\zeta>1-\eta$ we can choose $x\in D\cap H$ and $n\in \mathcal{H}(x,\sigma,U)$ such that
$f^n(D)$ contains the ball $\gamma:=B_{\xi}(f^n(x))$ of radius $\xi$ centered at $f^n(x)$, which satisfies
\begin{equation}\label{uu}
\frac{{\rm Leb}_{\gamma}(f^n(H))}{{\rm Leb}_{\gamma}(\gamma)}>\zeta.
\end{equation}
Since $n\in \mathcal{H}(x,\sigma,U)$, we know $f^n(x)\in U$ by definition. By construction of $U$, there exists $1\le i \le \ell$ such that
$$
f^n(x)\in B\quad \textrm{for some} ~B\in \F^i(r).
$$
According to the above Claim, we know this $\gamma$ also satisfies the estimate (\ref{ttg}).
Combining estimates (\ref{ttg}), (\ref{uu}) and the choices of $\eta$ and $\zeta$, we obtain 
$$
{\rm Leb}_{\gamma}(\B(\nu_i,f)\cap f^n(H))>0.
$$
Using the invariance of $\B(\nu_i,f)$, we then get
$$
{\rm Leb}_D(\B(\nu_i,f)\cap H)>0.
$$
This gives a contradiction to the definition of $H$.
\end{proof}

All that is left to do is to prove the Claim above.

\begin{proof}[Proof of the Claim]
For each $1\le i \le \ell$, for each $B\in \F^i(r)$, since $\nu_i(B\cap \Lambda_k)>0$ and $\nu_i$ is an ergodic Gibbs $E$-state, there exists $\gamma_B\in \mathcal{R}_B$ such that 
\begin{equation}\label{fde}
{\rm Leb}_{\gamma_B}(\B(\nu_i,f)\cap \Lambda_k)>0.
\end{equation}
Recall the choice of $r$, we know that for any smooth disk $\gamma\ni x$ tangent to $\mathscr{C}_a^{E}$ of radius $\xi$ around $x$, $\gamma$ cuts $\mathcal{F}_{\delta}^s(y)$ transversely for every $y\in \Lambda_k\cap \gamma_B$. 
Consequently, from (\ref{fde}) and the absolute continuity of Pesin stable lamination, one can take a uniform $\eta>0$ such that 
$$
\frac{{\rm Leb}_{\gamma}(\B(\nu_i,f))}{{\rm Leb}_{\gamma}(\gamma)}>\eta.
$$
for every $1\le i \le \ell$ and disk $\gamma$ described as above.
\end{proof}

\section{Proof of Theorem \ref{TheoB}}\label{pb}

Let us remark that in Theorem \ref{TheoB}, we do not involve any information on Lyapunov exponents along the first one-dimensional center $E_1^c$ in terms of empirical measures, this is mainly due to the following observation. A proof can be found in \cite[Lemma 3.4]{MC23}.

\begin{proposition}\label{ded}
Let $f\in \P_k(M)$ with partially hyperbolic splitting 
$$
TM=E^u\oplus_{\succ} E_1^c\oplus_{\succ} \cdots \oplus_{\succ} E_k^c\oplus_{\succ} E^s.
$$
If ${\rm Leb}(M\setminus \cup_{\nu\in \mathcal{G}_0}\B(\nu,f))>0$, then for Lebesgue almost every $x\in M\setminus \cup_{\nu\in \mathcal{G}_0}\B(\nu,f)$, any $\mu\in \mathscr{M}(x)$ admits only positive Lyapunov exponents along $E_1^c$.
\end{proposition}

\medskip
Now we give the proof of Theorem \ref{TheoB}.

\begin{proof}[Proof of Theorem \ref{TheoB}]
By Proposition \ref{ded}, one knows that for Lebesgue almost every point $x$ of $M\setminus \cup_{\nu\in \G_0} \B(\nu,f)$, any $\mu\in \mathscr{M}(x)$ admits only positive Lyapunov exponents along $E_1^c$. In other words, we have 
$$
\lambda_1^c(z)>0, \quad \mu \textrm{-a.e.}~ z\in M.
$$ 

\smallskip
Let us give a more accurate classification on full Lebesgue measure subset of $M\setminus \cup_{\nu\in \G_0}  \B(\nu,f)$. Indeed, for every $1\le i \le k$, consider $K_i$ as the set of points $x$ for which there exists $\mu\in \mathscr{M}(x)$ such that 
\begin{equation}\label{tt}
\lambda_i^c(z)>0, \quad \lambda_{i+1}^c(z)<0, \quad \mu \textrm{-a.e.}~ z\in M.
\end{equation}
By assumption, Lebesgue almost every point of $M\setminus \cup_{\nu\in \G_0}  \B(\nu,f)$
is contained in $K_i$ for some $1\le i \le k$. 

\smallskip
By Theorem \ref{TheoA}, $\cup_{0\le i \le k}\G_i$ is the set of all ergodic physical measures. As a result, to prove the result, it suffices to show the following: for every $1\le i \le k$, Lebesgue almost every point of $K_i$ is contained in $\cup_{\nu\in \G_i} \B(\nu,f)$. 

\smallskip
For each fixed $1\le i \le k$, we assume that ${\rm Leb}(\K_i)>0$ without loss of generality. Rewrite $E_i=E^u\oplus_{\succ}\cdots\oplus_{\succ} E_i^c$ and $F_i=E_{i+1}^c\oplus_{\succ}\cdots\oplus_{\succ} E^s$.
By Proposition \ref{cug} and the definition of $K_i$, there exists a full Lebesgue subset $\K_i$ of $K_i$ such that for every $x\in \K_i$, there exists $\mu\in \mathscr{M}(x)$ satisfying
$$
{\rm h}_{\mu}(f)\ge\int \log |{\rm det}Df|_{E_i}|{\rm d}\mu
$$
and (\ref{tt}).
This together with the Ruelle's inequality ensures that $\mu$ is a hyperbolic SRB measure of index ${\rm dim}E_i$.
We know $\#\G_i<+\infty$ by Theorem \ref{TheoA}, recall that $\G_i$ is the set of ergodic hyperbolic SRB measures of index ${\rm dim}E_i$. 
By applying Theorem \ref{TheoC} to the dominated splitting $TM=E_i\oplus_{\succ} F_i$, we conclude that Lebesgue almost every point of $K_i$ is contained in $\cup_{\nu\in \G_i} \B(\nu,f)$.
This completes the proof of Theorem \ref{TheoB}.
\end{proof}

\end{document}